\documentclass{article}

\usepackage[T1]{fontenc}
\usepackage[utf8]{inputenc}
\usepackage{amsmath,amssymb,mathtools,amsthm}
\usepackage[hidelinks]{hyperref}
\usepackage{tikz}
\newtheorem{theorem}{Theorem}[section]
\newtheorem{proposition}[theorem]{Proposition}
\newtheorem{lemma}[theorem]{Lemma}

\theoremstyle{remark}

\usepackage{pgfplots}
\pgfplotsset{compat=1.18}

\newcommand{\R}{\mathbb{R}}

\newcommand{\norm}[1]{\| #1\|}

\newcommand{\dd}{\,\mathrm{d}}
\title{
Nesterov Flow May Travel Infinitely Long to Converge to a Minimizer
}
\author{Ernest K. Ryu}
\date{}

\begin{document}

\maketitle

\begin{abstract}
Recent work has established that the trajectory of the Nesterov ODE, a the continuous-time model of Nesterov's accelerated gradient method, exhibits point convergence towards a minimizer of a convex potential. A natural next question is whether this point convergence can be upgraded to rectifiability, namely whether the convergent orbit has finite path length. This work provides the answer in the negative by constructing a differentiable convex potential in $\mathbb{R}^2$ for which the flow converges to its minimizer but still accumulates infinite path length. All proofs of this work are due entirely to an internal model at OpenAI.
\end{abstract}

\section{Introduction}

The critical Nesterov flow
\[
  \ddot X(t)+\frac{3}{t}\dot x(t)+\nabla f(x(t))=0
  \qquad t>0
\]
with initial condition $X(0)=X_0$ and $\dot{X}(0)=0$ and differentiable convex function $f\colon\mathbb{R}^d\rightarrow\mathbb{R}$ was introduced by Su, Boyd, and Cand\`es \cite{SuBoydCandes2014,SuBoydCandes2016} as the continuous-time model of Nesterov's accelerated gradient method \cite{Nesterov1983}. 
A Lyapunov analysis yields the function-value rate
\[
  f(X(t)) - \min f = \mathcal{O}(1/t^2),
\]
and a large body of literature has studied finer asymptotic properties of the Nesterov flow \cite{attouch2000heavy,chambolle2015convergence,attouch2015fast,attouch2016fast,may2017asymptotic,AttouchChbaniPeypouquetRedont2018,attouch2019rate,attouch2025recovering}. Most recently, Jang, Ryu, Bo{\c t}, Fadili, and Nguyen 
 proved that the trajectory itself converges to a minimizer if a minimizer exists \cite{JangRyu2025,BotFadiliNguyen2025}.

There is also an active body of work bounding the path length for gradient flow $\dot{X}=-\nabla f(X)$ based on the Kurdyka--{\L}ojasiewicz (KL) inequality \cite{Lojasiewicz1963,Kurdyka1998,BolteDaniilidisLeyMazet2010,AbsilMahonyAndrews2005,AttouchBolteRedontSoubeyran2010,AttouchBolteSvaiter2013,AttouchChbaniPeypouquetRedont2018,AttouchChbaniRiahi2019,DAcuntoKurdyka2021}
and the notion of self-contracting curves 
\cite{ManselliPucci1991,DaniilidisLeySabourau2010,DaniilidisDavidDurandCartagenaLemenant2015,StepanovTeplitskaya2017,GuptaBalakrishnanRamdas2021}.



For the critical Nesterov flow, a natural next question is whether this point convergence can be upgraded to \emph{rectifiability}, namely whether the convergent orbit has finite path length. This paper presents the answer in the negative. We construct a convex $C^1$ potential $f\colon \mathbb{R}^2\rightarrow\mathbb{R}$ for which the flow converges to its minimizer but still accumulates infinite path length.

The construction is based on two key insights. First, for radial potentials, the associated flow satisfies a conservation-of-angular-momentum-like identity:
\[
  t^3
  \!\!\!
\underbrace{
  \det\big(x(t),\dot x(t)\big)}_{=\text{angular momentum}}\!\!\equiv \kappa,.
\]
Second, the pathology arises when the potential behaves like $\|x\|$ near the origin. But, to ensure differentiability at the origin, we have the gradient vanish slowly (logarithmically) as $x\rightarrow 0$.

\paragraph{Comment on the use of AI.} The mathematical proofs in this paper are due entirely to an internal model at OpenAI. The role of the human author was simply to digest the proofs and modify the write-ups for clarity.


\section{Construction with infinite path length}

In this section, we prove the following principal claim.

\begin{theorem}\label{thm:main}
There exist a convex $C^1$ function $f\colon\R^2\to\R$ with unique minimizer at the origin and a global solution $X\colon[0,\infty)\to\R^2$ of
\[
   \ddot X(t)+\frac{3}{t}\dot X(t)+\nabla f(X(t))=0,
  \qquad t>0,
  \qquad X(0)=X_0,
  \qquad \dot X(0)=0,
\]
some $X_0\neq0$, 
such that
\[
  \int_0^\infty \| \dot X(t)\|\,\dd t=\infty.
\]
\end{theorem}

\subsection{Almost radial construction}

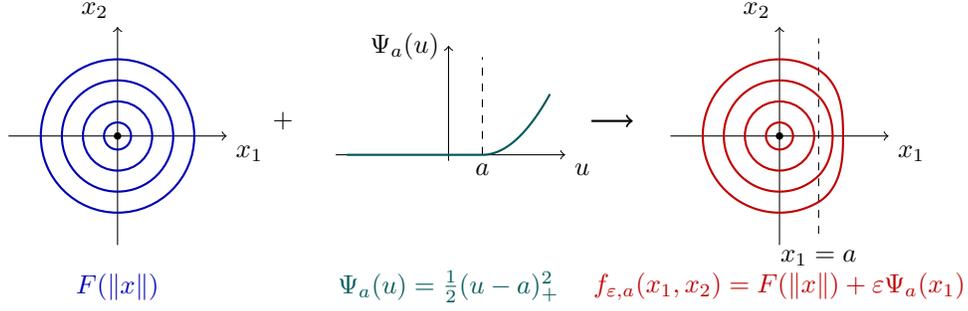
\begin{figure}[t]
\centering
\hspace{-0.3in}
\begin{tikzpicture}[scale=1]
  \begin{scope}[shift={(-4.4,0)}]
    \draw[->] (-1.45,0) -- (1.45,0) node[below right] {$x_1$};
    \draw[->] (0,-1.45) -- (0,1.45) node[above left] {$x_2$};
    \foreach \r in {0.18,0.46,0.74,1.02} {
      \draw[blue!70!black,thick] (0,0) circle (\r);
    }
    \fill (0,0) circle (1.4pt);
    \node[blue!70!black] at (0,-2.0) {$F(\|x\|)$};
  \end{scope}

  \begin{scope}[shift={(0,-0.25)}]
    \draw[->] (-1.5,0) -- (1.55,0) node[below right] {$u$};
    \draw[->] (0,-0.08) -- (0,1.45) node[left] {$\Psi_a(u)$};
    \draw[dashed] (0.45,0) -- (0.45,1.3);
    \node[below] at (0.45,0) {$a$};
    \draw[teal!70!black,thick] (-1.35,0) -- (0.45,0);
    \draw[teal!70!black,thick,samples=120,domain=0.45:1.35,smooth]
      plot (\x,{1.0*(\x-0.45)^2});
    \node[teal!70!black] at (0,-1.75) {$\Psi_a(u)=\frac12(u-a)_+^2$};
  \end{scope}

  \begin{scope}[shift={(4.4,0)}]
  \def\a{0.52}
  \draw[->] (-1.45,0) -- (1.45,0) node[below right] {$x_1$};
  \draw[->] (0,-1.45) -- (0,1.45) node[above left] {$x_2$};
  \draw[dashed] (\a,-1.3) -- (\a,1.3);
  \node[below] at (\a,-1.4) {$x_1=a$};

  \foreach \rr in {0.18,0.46} {
    \draw[red!75!black,thick] (0,0) circle (\rr);
  }

  \foreach \rr in {0.74,1.02} {
    \draw[red!75!black,thick,samples=200,smooth,domain=0:360,variable=\t]
      plot ({(\rr - 0.7*(max(\rr*cos(\t)-\a,0))^2)*cos(\t)},
            {(\rr - 0.7*(max(\rr*cos(\t)-\a,0))^2)*sin(\t)});
  }

  \fill (0,0) circle (1.4pt);
  \node[red!75!black] at (0,-2.0)
    {$f_{\varepsilon,a}(x_1,x_2)=F(\|x\|)+\varepsilon\Psi_a(x_1)$};
\end{scope}

  \node at (-2.2,0.2) {$+$};
  \draw[->,thick] (1.9,0.2) -- (2.45,0.2);
\end{tikzpicture}
\caption{Illustration of the potential $f_{\varepsilon,a}$. The function is radial near the origin, while the one-sided quadratic perturbation $\varepsilon\Psi_a(x_1)$, whose purpose is to generate positive angular momentum, becomes inactive once the trajectory is sufficiently close to the origin.
}
\label{fig:feps-schematic}
\end{figure}

Define $F\colon[0,\infty)\to\R$ by
\[
  F(0)=0,
  \qquad
  F(r)=
  \begin{cases}
    \displaystyle \int_0^r \frac{\dd u}{-\log u}, & 0<r\le e^{-2},\\[6pt]
    F(e^{-2})+\tfrac12\bigl(r-e^{-2}\bigr), & r\ge e^{-2}.
  \end{cases}
\]
Let $a>0$ and $\varepsilon>0$. Define
\[
  \Psi_a(u)=\frac12 (u-a)^2_+,
  \qquad
  (u-a)_+=\max\{u-a,0\}
  \]
and
\[
   f_{\varepsilon,a}(x_1,x_2)
  =
  F\Big(\sqrt{x_1^2+x_2^2}\Big)+\varepsilon \Psi_a(x_1).
\]
See Figure~\ref{fig:feps-schematic} for an illustration of $f_{\varepsilon,a}$.
From direct calculations, it can be checked that $f_{\varepsilon,a}\colon\mathbb{R}^2\rightarrow\mathbb{R}$ is convex and continuously differentiable, and $0$ is its unique minimizer of $f_{\varepsilon,a}$.

Let $\{X(t)\}_{t\ge 0}$ be the solution to
\[
   \ddot X(t)+\frac{3}{t}\dot X(t)+\nabla f_{\varepsilon,a}(X(t))=0,
  \qquad t>0,
  \qquad X(0)=(2a,a),
  \qquad \dot X(0)=0.
\]
Since the minimizer is unique, the convergence results of prior work imply $X(t)\rightarrow 0$. 
We refer to 
\[
\{x\in \mathbb{R}^2:\|x\|\le a\}
\]
as the \emph{radial region} since $f_{\varepsilon,a}(x)=F(\|x\|)$ is radial in this region. 
Note that the initial point $X(0)=(2a,a)$ is outside of the radial region.
Let
\begin{equation*}
T_\mathrm{rad}=\sup\{t: \|X(t)\|> a\}
\end{equation*}
be the time at which the dynamics $\{X(t)\}_{t\ge 0}$ permanently enters the radial region. Since $X(t)\rightarrow 0$, we have $T_\mathrm{rad}<\infty$.

\begin{figure}
\begin{tikzpicture}
\hspace{-0.65in}
\begin{axis}[
  name=orbit,
  width=8.0cm,
  height=6.8cm,
  axis equal image,
  title={Trajectory in $\mathbb{R}^2$},
  xmin=-0.044,
  xmax=0.044,
  ymin=-0.044,
  ymax=0.044,
  xtick={-0.04,-0.02,0,0.02,0.04},
  ytick={-0.04,-0.02,0,0.02,0.04},
  scaled ticks=false,
  tick label style={/pgf/number format/fixed,/pgf/number format/precision=2},
  grid=both,
]
  \addplot[very thick, color=red!75!black]
    table[x=x1, y=x2, col sep=tab] {nesterov_static_orbit_eps50.tsv};
  \addplot[dashed, color=orange!85!black, domain=-0.044:0.044, samples=2]
    ({0.02}, {x});
  \addplot[only marks, mark=*, mark size=1.6pt, color=black]
    coordinates {(0,0)};
  \addplot[only marks, mark=*, mark size=1.9pt, color=violet]
    coordinates {(0.04, 0.02)};
\end{axis}

\begin{axis}[
  at={(orbit.south east)},
  anchor=south west,
  xshift=1.8cm,
  width=8.4cm,
  height=6.8cm,
  name=growthleft,
  xmode=log,
  ymode=log,
  xlabel={Time $t$},
  ylabel style={color=blue!75!black},
  yticklabel style={text=blue!75!black},
  axis y line*=left,
  axis x line*=bottom,
  title={Function value and accumulated path length},
  grid=both,
  legend style={at={(0.13,0.86)}, anchor=south west},
  xmin=1,
  xmax=100000,
]
  \addplot[very thick, color=blue!75!black]
    table[x=time, y=f_value, col sep=tab] {nesterov_static_series_eps50.tsv};
  \addlegendentry{$f(X(t))$}
\end{axis}

\begin{axis}[
  at={(growthleft.south west)},
  anchor=south west,
  width=8.4cm,
  height=6.8cm,
  xmode=log,
  ymode=normal,
  xmin=1,
  xmax=100000,
  ymin=0,
  ymax=0.2,
  ytick={0.05,0.1,0.15,0.2},
  yticklabels={0.05,0.1,0.15,0.2},
  scaled y ticks=false,
  ylabel style={color=red!75!black},
  yticklabel style={text=red!75!black},
  axis y line*=right,
  axis x line=none,
  xticklabels={},
  grid=none,
  legend style={at={(0.98,0.52)}, anchor=north east},
]
  \addplot[very thick, color=red!75!black]
    table[x=time, y=arclength, col sep=tab] {nesterov_static_series_eps50.tsv};
  \addlegendentry{accumulated path length}
\end{axis}
\end{tikzpicture}

\caption{Numerical simulation of the pathological Nesterov trajectory with $a=0.02$ and $\varepsilon=50$.}
\label{fig:illustration}
\end{figure}
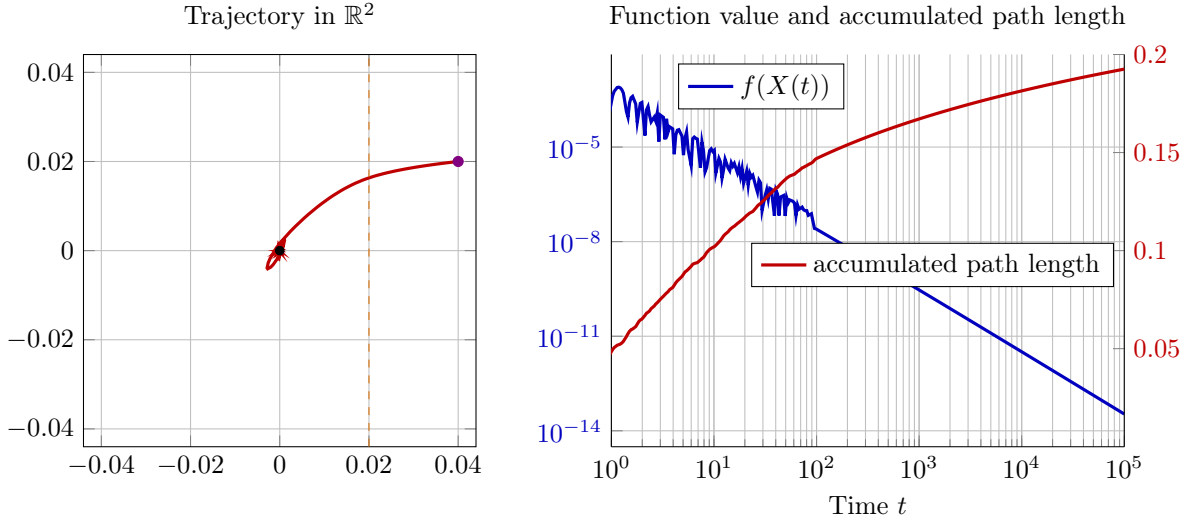

\subsection{Creation and conservation of angular momentum}

Consider the angular momentum
\[
    J (t):=\det\bigl(X (t),\dot X (t)\bigr).
\]
In the ODE $\ddot X(t)+\frac{3}{t}\dot X(t)+\nabla f_{\varepsilon,a}(X(t))=0$, the friction coefficient $3/t$ dissipates angular momentum. However, the quantity $t^3J (t)$ is conserved once the dynamics enters the radial region.

\begin{lemma}\label{lem:J-balance}
For all $t>0$,
\[
    t^3J (t)
  =
  \varepsilon \int_0^t s^3 X_{2}(s)\,\bigl(X_{1}(s)-a\bigr)_+\,\dd s,
\]
where we use the notation $X(t)=(X_1(t),X_2(t))\in \mathbb{R}^2$.
In particular, for all $t\ge T_\mathrm{rad}$,
\[
    t^3J (t)=\kappa_\varepsilon,
  \qquad
  \kappa_\varepsilon:=
  \varepsilon \int_0^{T_\mathrm{rad}} s^3 X_{2}(s)(X_{1}(s)-a)_+\,\dd s.
\]
\end{lemma}

\begin{proof}
Differentiating $J $, we get
\begin{align*}
  J '(t)
  &=
  \det\bigl(\dot X (t),\dot X (t)\bigr)
  +\det\bigl(X (t),\ddot X (t)\bigr)\\
  &=
  \det\!\Big(
    X (t),
    -\frac{3}{t}\dot X (t)-\nabla f_{\varepsilon,a}(X (t))
  \Big).
\end{align*}
Now the radial part $x\mapsto F(\|x\|)$ contributes no torque, since its gradient is parallel to $x$. For the perturbation,
\[
  \nabla\bigl(\varepsilon\Psi_a(x_1)\bigr)
  =
  \bigl(\varepsilon(x_1-a)_+,0\bigr),
\]
and therefore
\[
  \det\Bigl((x_1,x_2),\nabla(\varepsilon\Psi_a(x_1))\Bigr)
  =
  -\varepsilon x_2(x_1-a)_+.
\]
Hence
\[
  J '(t)
  =
  -\frac{3}{t}J (t)
  +
  \varepsilon\,X_{2}(t)(X_{1}(t)-a)_+.
\]
Multiplying by $t^3$, integrating, and using $J (0)=0$ yields the first claim.
Since $(X_{1}(t)-a)_+=0$ for $t\ge T_\mathrm{rad}$, the right-hand side becomes constant for $t\ge T_\mathrm{rad}$, proving the second claim.
\end{proof}

\begin{proposition}\label{prop:created-angular-momentum}
For a sufficiently small $\varepsilon>0$, 
\[
  \kappa_\varepsilon:=
  \varepsilon \int_0^{T_\mathrm{rad}} s^3 X_{2}(s)(X_{1}(s)-a)_+\,\dd s>0.
\]
\end{proposition}
(It can be shown that $\kappa_\varepsilon>0$ holds for all $\varepsilon>0$, but we only need the result for one value of $\varepsilon>0$, and the small-$\varepsilon$ proof is shorter.)
\begin{proof}
Consider the ODE
\[
  \ddot X_\varepsilon(t)+\frac{3}{t}\dot X_\varepsilon(t)+\nabla f_{\varepsilon,a}(X_\varepsilon(t))=0,
  \quad t>0,
  \quad X_\varepsilon(0)=(2a,a),
  \quad \dot X_\varepsilon(0)=0.
\]
where we denote the dependence on $\varepsilon$ explicitly. When $\varepsilon=0$, the equation is radial, so the solution stays on the ray through $X_0=(2a,a)$. Thus
\[
  X_0(t)=\rho(t)
  \begin{bmatrix}
  2\\1
  \end{bmatrix}
\]
for some continuous scalar function $\rho(t)$ such that $\rho(0)=a$. Consequently,
\[
  X_{0,2}(s)(X_{0,1}(s)-a)_+
  =
  \rho(s)(2\rho(s)-a)_+
\]
is strictly positive for sufficiently small $s>0$
and nonnegative for all $s$, since $\rho(s)<0$ implies $(2\rho(s)-a)_+=0$.
Therefore
\[
  \int_0^T s^3 X_{0,2}(s)(X_{0,1}(s)-a)_+\,\dd s>0
\]
for any $T>0$.

 By continuous dependence of solutions on the parameter $\varepsilon$, for sufficiently small $\varepsilon$, we have 
\[
  \int_0^{T_\mathrm{rad}} s^3 X_{\varepsilon,2}(s)(X_{\varepsilon,1}(s)-a)_+\,\dd s>0
\]
remains true for all sufficiently small $\varepsilon>0$.
\end{proof}

Proposition~\ref{prop:created-angular-momentum} guarantees existence of an $\varepsilon>0$ that generates strictly positive angular momentum. Fix such an $\varepsilon$ for the remainder of the proof and set
\[
  f:=f_{\varepsilon,a},
  \qquad
  \kappa:=\kappa_\varepsilon.
\]

\subsection{Asymptotic dynamics and infinite path length}

Let
\[
   r(t)=\|X(t)\|
\]
for $t\ge 0$. By the convergence results of prior work, $r(t)\rightarrow 0$.
We first note a convenient lower bound on $F$.
\begin{lemma}\label{lem:lower-bound-F}
There exists $c_F>0$ such that
\[
  F(r)\ge c_F\,\frac{r}{-\log r}
  \qquad\text{for all }0<r\le e^{-2}.
\]
\end{lemma}

\begin{proof}
For $0<r\le e^{-2}$, apply a change of varaibles for the integral defining $F$ and get
\[
  F(r)=\int_R^\infty \frac{e^{-v}}{v}\,\dd v,
\]
where $R=-\log r$.
Then,
\[
  F(r)
  \ge \int_R^{R+1}\frac{e^{-v}}{v}\,\dd v
  \ge \frac{e^{-(R+1)}}{R+1}
  =\frac{r}{e(R+1)}
  \ge \frac{1}{2e}\,\frac{r}{R}
  =\frac{1}{2e}\,\frac{r}{-\log r}.
\]
\end{proof}

Next, we show the rate of $r(t)\rightarrow 0$ is rather fast.
(In the quadratic potential case considered in Section~\ref{sec:quadratic}, it can be shown that $r(t)=\Theta(1/t^{3/2})$.)

\begin{proposition}\label{prop:radial-bound}
There exist constants $C_r>0$ and $T_r>0$ such that
\[
  r(t)\le C_r\,\frac{\log t}{t^2}
  \qquad\text{for all }t\ge T_r.
\]
\end{proposition}

\begin{proof}
Since $r(t)\to0$, so there exists $T_r\ge e$ such that $r(t)\le e^{-2}$ for all $t\ge T_r$. By definition of $f=f_{\varepsilon,a}$ and the $\mathcal{O}(1/t^2)$ bound of prior work \cite[Theorem~3]{SuBoydCandes2016}, we have 
\[
  F(r(t))\le f(X(t))\le \frac{2\|X(0)\|^2}{t^2}= \frac{10a^2}{t^2}.
\]
If $r(t)\le t^{-2}$, then automatically
\[
  r(t)\le \frac{\log t}{t^2}.
\]
Else, assume $r(t)>t^{-2}$. Then $-\log r(t)<2\log t$, and Lemma~\ref{lem:lower-bound-F} gives
\[
  \frac{c_F}{2\log t}\,r(t)
  \le c_F\,\frac{r(t)}{-\log r(t)}
  \le F(r(t))
  \le \frac{10a^2}{t^2}.
\]
Hence
\[
  r(t)\le \frac{20a^2}{c_F}\,\frac{\log t}{t^2}.
\]
Combining the two cases proves the claim.
\end{proof}

Finally, we establish that the path length is infinite and prove Theorem~\ref{thm:main}.

\begin{proof}[Proof of Theorem~\ref{thm:main}]
Since $\kappa\neq0$, the conservation law $t^3J(t)=\kappa\ne 0$ implies that $X(t)\neq0$ for all $t>T_\mathrm{rad}$. We may therefore decompose $\dot X(t)$ into radial and tangential components. Let
\[
  e_r(t)=\frac{X(t)}{r(t)}
\]
and choose a unit vector $e_\perp(t)$ orthogonal to $e_r(t)$ so that $\det(e_r(t),e_\perp(t))=1$. Then
\[
  \det\bigl(X(t),\dot X(t)\bigr)
  =
  r(t)\,\bigl(\dot X(t)\cdot e_\perp(t)\bigr).
\]
Hence, by the conservation law $t^3J(t)=\kappa$, 
\begin{equation}
\label{eq:conservation-law}  
  \| \dot X_\perp(t)\|
  :=
  \bigl\| \dot X(t)\cdot e_\perp(t)\bigr\|
  =
  \frac{\kappa}{t^3r(t)}.
\end{equation}
Proposition~\ref{prop:radial-bound} yields, for $t\ge \max\{T_r,T_\mathrm{rad}\}$,
\[
  \| \dot X_\perp(t)\|
  \ge \frac{\kappa}{C_r\,t\log t}.
\]
Since $\|\dot X(t)\|\ge\|\dot X_\perp(t)\|$, we conclude
\[
  \int_0^\infty \| \dot X(t)\|\,\dd t
  \ge \int_{\max\{T_r,T_\mathrm{rad}\}}^\infty \| \dot X_\perp(t)\|\,\dd t
  \ge \frac{\kappa}{C_r}\int_{\max\{T_r,T_\mathrm{rad}\}}^\infty \frac{\dd t}{t\log t}
  =\infty.
\]
Thus the curve is not rectifiable.
\end{proof}



\section{Further discussions}

\subsection{Rectifiability with quadratic potentials}\label{sec:quadratic}

Upon seeing the construction in Theorem~\ref{thm:main}, one may wonder whether a simpler quadratic example could exhibit the same nonrectifiable behavior. In this section, we show that this cannot occur for pure quadratic potentials.

Let
\[ 
   q(x)=\frac12 x^\top Qx,
\]
where $Q\in\R^{d\times d}$ is symmetric positive semidefinite, and consider 
\[
  \ddot X(t)+\frac{3}{t}\dot X(t)+QX(t)=0,
  \qquad t>0,
  \qquad X(0)=X_0,
  \qquad \dot X(0)=0.
\]
Let
 $Q=U^\top \Lambda U$, where $U$ is orthogonal and $\Lambda=\operatorname{diag}(\lambda_1,\dots,\lambda_d)$ with $\lambda_i\ge0$. Writing
\[
  y(t)=UX(t),
  \qquad
  a=UX_0,
\]
we obtain the decoupled scalar equations
\[
  \ddot y_i(t)+\frac{3}{t}\dot y_i(t)+\lambda_i y_i(t)=0,
  \qquad y_i(0)=a_i,
  \qquad \dot y_i(0)=0
\]
for $i=1,\dots,d$.

Let $J_\alpha$ denote the Bessel function of the first kind of order $\alpha$.
From the \cite[Section~3.2]{SuBoydCandes2016}, we know that the solutions are given by
\[y_i(t)=
  \begin{cases}
    a_i, & \lambda_i=0,\\[4pt]
    \displaystyle 2a_i\,\frac{J_1(\sqrt{\lambda_i}\,t)}{\sqrt{\lambda_i}\,t}, & \lambda_i>0.
  \end{cases}
\]
In the case $\lambda_i>0$ one also has
\[
    \dot y_i(t)=-\frac{2a_i}{t}J_2(\sqrt{\lambda_i}\,t),
  \qquad t>0.
\]
Then, we conclude the curve is rectifiable:
\begin{align*}
  \int_0^\infty \|\dot X(t)\|\,\dd t
  &=\int_0^\infty \| \dot y(t)\|\,\dd t
  \le \sum_{\lambda_i>0}\int_0^\infty \|\dot y_i(t)\|\,\dd t\\
  &\le\sum_{\lambda_i>0}\| a_i\|\int_0^\infty \frac{\| J_2(\sqrt{\lambda_i}t)\|}{t}\,\dd t\\
  &=2\sum_{\lambda_i>0}\| a_i\|\int_0^\infty \frac{\| J_2(s)\|}{s}\,\dd s<\infty.
\end{align*}

\subsection{Gradient flow is radial, rectifiable, and reaches the minimizer in finite time}\label{sec:gradient-flow}

Upon seeing the pathological dynamics of Nesterov flow, one may wonder whether \emph{gradient flow} $\dot{X}=-\nabla f(X)$ on the potential $f_{\varepsilon,a}$ also exhibits similarly pathological behavior. In this section, we argue that it does not: gradient-flow trajectories have finite path length and reach the minimizer in \emph{finite time}.

At a high level, the argument is as follows. Since gradient flow converges, the trajectory eventually enters the radial region. Once inside that region, the potential is radial, so the trajectory evolves along a straight ray toward the origin; this immediately yields rectifiability. The remaining dynamics reduce to a one-dimensional ODE for the radius, whose explicit analysis shows finite-time arrival at the minimizer.

This comparison highlights the point that while the momentum-driven Nesterov acceleration does improve worst-case function-value guarantees, but it may slow down convergence and cause extreme oscillatory behavior in some instances.
In other words, Nesterov flow may be worse than gradient flow for some functions.

\subsection{Further pathological properties of the dynamics}


In this section, we further analyze the dynamics induced by $f=f_{\varepsilon,a}$ and obtain further insights on the worst-case behavior of Nesterov flow.

We first review the standard Lyapunov analysis for the critical Nesterov flow.
\begin{lemma}\label{lem:energy-eps}
Consider the setup of Theorem~\ref{thm:main}.
The Lyapunov function
\begin{equation*}
  \mathcal{E}(t)=
  t^2 f_{\varepsilon,a}(X (t))
  +\frac12\|2X (t)+t\dot X (t)\|^2.
\end{equation*}
is nonincreasing on $[0,\infty)$. Consequently,
\[
    \mathcal{E}(t)\le \mathcal{E}(0)=2\| X_0\|^2
  \qquad\text{for all }t\ge0.
\]
\end{lemma}
\begin{proof}
Follows from the analysis of Theorem 3 of \cite{SuBoydCandes2016}.
\end{proof}

Next, we show that the function value $f(X(t))$ actually matches the $\mathcal{O}(1/t^2)$ upper bound up to logarithmic factors. To the best of our knowledge, this is the first construction in which Nesterov flow exhibits an actual convergence rate slower than $\Theta(1/t^3)$, which is the rate achieved by quadratic potentials \cite[Section 3.2]{SuBoydCandes2016}.

\begin{theorem}\label{prop:objective-bounds}
In the setup of Theorem~\ref{thm:main}, 
there exist constants $c_0,C_0,T_0>0$ such that
\[
  \frac{c_0}{t^2\log t}
  \le f(X(t))
  \le \frac{C_0}{t^2}
  \qquad\text{for all }t\ge T_0.
\]
\end{theorem}

\begin{proof}
The upper bound follows immediately from prior work \cite[Theorem~3]{SuBoydCandes2016}.
We next obtain a lower bound on $r(t)$. By Lemma~\ref{lem:energy-eps},
\[
  t\|\dot X(t)\|
  \le \| 2X(t)+t\dot X(t)\|+2\| X(t)\|
  \le \sqrt{20a^2}+2\| X(t)\|.
\]
Because $X(t)\to0$ as $t\to\infty$, there exists a constant $C_v>0$ such that
\[
   t\|\dot X(t)\|\le C_v
  \qquad\text{for all }t\ge0.
\]
By \eqref{eq:conservation-law},
\[
  \frac{\kappa}{t^3r(t)}
  =\|\dot X_\perp(t)\|
  \le \|\dot X(t)\|
  \le \frac{C_v}{t},
  \qquad\text{for all }t\ge T_\mathrm{rad}.
\]
So,
\[
    r(t)\ge \frac{\kappa}{C_v\,t^2},
  \qquad\text{for all }t\ge T_\mathrm{rad}.
\]
Since $r(t)\rightarrow 0$, we may choose $T_0$ large enough that 
\[
  r(t)\le e^{-2}\quad\text{and}\quad  t\ge \max\Bigl\{e,\,\frac{C_v}{\kappa},\,T_\mathrm{rad}\Bigr\}
  \qquad\text{for all }t\ge T_0.
\]
Then with Lemma~\ref{lem:lower-bound-F}, we conclude 
\[
  f(X(t))
  \ge F(r(t))
  \ge c_F\,\frac{r(t)}{-\log r(t)}
  \ge c_F\,\frac{\kappa/(C_v t^2)}{2\log t-\log(\kappa/C_v)}
  \ge \frac{c_F\kappa}{3C_v\,t^2\log t}
\]
for $t\ge T_0$.
This proves the lower inequality with $c_0=c_F\kappa/(3C_v)$.
\end{proof}

Prior work on point convergence for Nesterov-type accelerated methods took the route of proving finiteness of integrals
\[
  \int_0^\infty t\,\big( f(X(t))-\min f\big)\,\dd t=\infty,
  \qquad
  \int_0^\infty t\,\norm{\dot X(t)}^2\,\dd t=\infty.
\]
By contrast, \cite{JangRyu2025,BotFadiliNguyen2025} established point convergence without such bounds, and suggested that these integrals may be infinite even when convergence holds. The following result provides closure to this question: point convergence does not require their finiteness. This also explains that the earlier attempts to prove point convergence via such estimates were pursuing an unviable path.

\begin{theorem}\label{thm:weighted-divergence}
In the setup of Theorem~\ref{thm:main}, 
\[
  \int_0^\infty t\,f_{\varepsilon,a}(X(t))\,\dd t=\infty,
  \qquad
  \int_0^\infty t\,\norm{\dot X(t)}^2\,\dd t=\infty.
\]
\end{theorem}


\begin{proof}
By Proposition~\ref{prop:objective-bounds},
\[
  f_{\varepsilon,a}(X(t))\ge \frac{c_0}{t^2\log t}
  \qquad\text{for all }t\ge T_0.
\]
Therefore
\[
  \int_0^\infty t\,f_{\varepsilon,a}(X(t))\,\dd t
  \ge
  \int_{T_0}^\infty t\,f_{\varepsilon,a}(X(t))\,\dd t
  \ge
  c_0\int_{T_0}^\infty \frac{\dd t}{t\log t}
  =\infty.
\]
This proves the first claim.

For the second claim, define
\[
  \mathcal{K}(t):=t^2 f(X(t))+\frac12 t^2\norm{\dot X(t)}^2,
  \qquad t\ge 0.
\]
With direct calculations, we get
\begin{align*}
  \mathcal{K}'(t)
  &=2t\big(f(X(t))-\norm{\dot X(t)}^2\big)
    +t^2\Bigl\langle \dot X(t),\ddot X(t)+\tfrac3t\dot X(t)+\nabla f(X(t))\Bigr\rangle \\
  &=2t\big(f(X(t))-\norm{\dot X(t)}^2\big)
  \qquad\text{for all }t>0.
\end{align*}
Hence, for all $R\ge T_0$,
\[
  \underbrace{
  \int_{T_0}^R t\,\norm{\dot X(t)}^2\,\dd t}_{\rightarrow\infty \text{ as } R\rightarrow\infty}
  =
  \underbrace{
  \int_{T_0}^R t\,f(X(t))\,\dd t}_{\rightarrow \infty\text{ as } R\rightarrow\infty}
  +\underbrace{\frac{\mathcal{K}(T_0)-\mathcal{K}(R)}{2}}_{\text{bounded as } R\rightarrow\infty}.
\]
It remains to show that $\mathcal{K}$ is bounded on $[T_0,\infty)$. The quantity
\[
  \mathcal{E}(t)=t^2 f(X(t))+\frac12 \norm{2X(t)+t\dot X(t)}^2
\] 
is bounded by Lemma~\ref{lem:energy-eps}, $X(t)\rightarrow 0$ implies $X(t)$ is bounded, so
$\norm{t\dot X(t)}$ is bounded. Consequently, $\mathcal{K}$ is bounded on $[T_0,\infty)$.
\end{proof}

\section{Conclusion}
This work identifies a new pathological property of the momentum-driven Nesterov flow: trajectories may have infinite path length despite the point convergence. Together with the parallel line of work bounding the path length of gradient flow, this result highlights path length as a regularity property of interest in the analysis of optimization algorithms.

\end{document}